\documentclass[a4paper,11pt]{article}
\usepackage[latin1]{inputenc}
 \usepackage[T1]{fontenc}
  \usepackage{amsmath,amsthm,amssymb}
 \usepackage[all]{xy}

\def \N{{\mathbb N}}

\def \R{{\mathbb R}}

\def \1{{\mathbb 1}}

\newtheorem{Def}{Definition}

\newtheorem{Lem}{Lemma}

\newtheorem{Def Nota}{Definitions and notations} 
\newtheorem{Cor}{Corollary}

\font\ninerm=cmr9
\long\outer\def\abstract#1{\bigskip\vbox{\noindent\ninerm
\baselineskip=10pt#1}\nobreak\bigskip}
\catcode`\â= \active \def â{\^a}
\catcode`\ê= \active \def ê{\^e}
\catcode`\ô= \active \def ô{\^o}
\catcode`\û= \active \def û{\^u}
\catcode`\î= \active \def î{{\^\i}}
\catcode`\é= \active \def é{\'e}
\catcode`\è= \active \def è{\`e}
\catcode`\à= \active \def à{\`a}
\catcode`\ù= \active \def ù{\`u}
\catcode`\ç= \active \def ç{\c {c}}
\catcode`\ï= \active \def ï{{\"\i}}

\newcount\numero
\def\exo#1{\advance\numero by 1\bigskip
{\noindent\tenbf #1\the\numero. }}
\def\frac#1#2{{#1\over #2}}
\title{A usufel lemma for Lagrange multiplier rules in infinite dimension.}   
\author{Mohammed Bachir and Jo\"{e}l Blot}
\begin{document}
\maketitle
\begin{center} {\it Laboratoire SAMM 4543, Université Paris 1 Panthéon-Sorbonne, Centre P.M.F. 90 rue Tolbiac 75634 Paris cedex 13}
\end{center}
\begin{center} 
{\it Email : Mohammed.Bachir@univ-paris1.fr}\\
\end{center}
\begin{center} {\it Laboratoire SAMM 4543, Université Paris 1 Panthéon-Sorbonne, Centre P.M.F. 90 rue Tolbiac 75634 Paris cedex 13}
\end{center}
\begin{center}
{\it Email : blot@univ-paris1.fr}\\
\end{center}
\noindent\textbf{Abstract.} We give some reasonable and usable conditions on a sequence of norm one in a dual banach space under which the sequence does not converges to the origin in the $w^*$-topology. These requirements help to ensure that the Lagrange multipliers are nontrivial, when we are interested for example on the infinite dimensional infinite-horizon Pontryagin
Principles for discrete-time problems.

\vskip3mm
\noindent
{\bf Keyword, phrase:} Baire category theorem, Subadditive and continuous map, Multiplier rules.\\
{\bf 2010 Mathematics Subject: 54E52, 49J21.} 
\section{Introduction.} 
Let $Z$ be a Banach space and $Z^*$ its topological dual. It is well known that in infinite dimensional separable Banach space, it is always true that the origin in $Z^*$ is the $w^*$-limit of a sequence from the unit sphere $S_{Z^*}$ as it is in its $w^*$-closure. In this paper, we look about reasonable and usable conditions on a sequence of norm one in $Z^*$ such that this sequence does not converge to the origin in the $w^*$-topology. This situation has the interest, when we are looking for a nontrivial Lagrange multiplier for optimization problems, and was encountered several times in the literature. See for example \cite{BB} and \cite{MM}. To guarantee that the multiplier are nontrivial at the limit, the authors in \cite{MM} used the following lemma from [\cite{LY}, pp. 142, 135].

\begin{Def} \label{D} A subset $Q$ of a Banach space $Z$ is said to be of finite codimension in $Z$ if there exists a point $z_0$ in the closed convex hull of $Q$ such that the closed vector space generated by $Q- z_0:= \left\{q - z_0 | \hspace{1mm} q \in Q\right\}$ is of finite codimension in $Z$ and the closed convex hull of $Q-z_0$ has a no empty interior in this vector space.
\end{Def}
 
\begin{Lem} (\cite{LY}, pp. 142, 135) Let $Q\subset Z$ be a subset of finite codimension in $Z$. Let $(f_k)_k\subset Z^*$ and $\epsilon_k\geq 0$ and $\epsilon_k\rightarrow 0$ such that 
\begin{itemize}
\item[$i)$] $\|f_k\|\geq\delta> 0$, for all $k\in \N$ and $f_k\stackrel{w^*}{\rightarrow} f$.
\item[$ii)$] for all $z\in Q$, and for all $k\in \N$, $f_k(z)\geq -\epsilon_k $.
\end{itemize}
Then, $f\neq 0$.
\end{Lem}

Note that, this is not the most general situation. Indeed, one can meet as in \cite{BB}, a situation where the part $ii)$ of the above lemma is not uniform on $z\in Z$, and depends on other parameter as follows: for all $z\in \overline{co}(Q)$, there exists $C_z\in \R$ such that for all $k\in \N$, $f_k(z)\geq -\epsilon_k C_z$. The principal Lemma \ref{Baire1} that we propose in this paper, will permit to include this very useful situation. This lemma is based on the Baire category theorem.
 
\section{Preliminary Lemmas.}

We need the following classical lemma. We denote by $Int(A)$ the topological interior of a set $A$.
\begin{Lem} \label{Connue} Let $C$ be a convex subset of a normed vector space. Let $x_0\in Int(C)$ and $x_1\in \overline{C}$. Then, for all $\alpha\in ]0,1]$, we have $\alpha x_0+(1-\alpha)x_1\in Int(C)$. 
\end{Lem} 
We deduce the following lemma.
\begin{Lem} \label{Connue1} Let $(F,\|.\|_F)$ be a Banach space and $C$ be a closed convex subset of $F$ with non empty interior. Suppose that $D\subset C$ is a closed subset of $C$ with no empty interior in $(C,\|.\|_F)$ (for the topology induced by $C$). Then, the interior of $D$ is non empty in $(F,\|.\|_F)$. 
\end{Lem} 

\begin{proof} On one hand, there exists $x_0$ such that $x_0\in Int(C)$. On the other hand, since $D$ has no empty interior in $(C,\|.\|_{F})$, there exists $x_1\in D$ and $\epsilon_1>0$ such that $B_F(x_1,\epsilon_1)\cap C\subset D$. By using Lemma \ref{Connue}, $\forall \alpha\in ]0,1]$, we have $\alpha x_0+(1-\alpha)x_1\in Int(C)$. Since $\alpha x_0+(1-\alpha)x_1\rightarrow x_1$ when $\alpha\rightarrow 0$, then there exist some small $\alpha_0$ and an integer number $N\in \N^*$ such that $B_F(\alpha_0 x_0+(1-\alpha_0)x_1,\frac{1}{N})\subset B(x_1,\epsilon_1)\cap C\subset D$. Thus $D$ has a non empty interior in $F$.
\end{proof}

\section{The principal Lemma.}
We give now our principlal lemma. We denote by $\overline{co}(X)$ the closed convex hull of $X$.

\begin{Lem} \label{Baire1} Let $Z$ be a Banach space. Let $(p_n)_n$ be a sequence of subadditive and continuous map on $Z$ and $(\lambda_n)_n\subset \R^+$ be a sequence of nonegative real number such that $\lambda_n\rightarrow 0$. Let $A$ be a non empty subset of $Z$, $a\in \overline{co}(A)$ and $F:=\overline{span(A-a)}$ the closed vector space generated by $A$. Suppose that $\overline{co}(A-a)$ has no empty interior in $F$ and that

\begin{itemize}
\item[$(1)$] for all $z\in \overline{co}(A)$, there exists $C_z \in \R$ such that for all $n\in \N$:
 $$p_n(z)\leq C_z \lambda_n.$$
\item[$(2)$] for all $z\in F$, $\limsup_n p_n(z)\leq 0$.
\end{itemize}

Then, for all bounded subset $B$ of $F$, we have   
\begin{eqnarray} 
\limsup_n \left(\sup_{h\in B} p_n(h) \right) &\leq& 0.\nonumber
\end{eqnarray}
\end{Lem} 
\begin{proof} For each $m\in \N$, we set
$$F_m:=\left\{z\in \overline{co}(A):  p_n(z)\leq m \lambda_n, \forall n\in \N\right\}.$$ 
The sets $F_m$ are closed subsets of $Z$. Indeed, 
$$F_m=\left(\bigcap_{n\in \N} p_n^{-1}(]-\infty, m \lambda_n])\right)\cap \left(\overline{co}(A)\right)$$
where, for each $n\in \N$, $p_n^{-1}(]-\infty, m \lambda_n])$ is a closed subset of $Z$ by the continuity of $p_n$. On the other hand, we have $\overline{co}(A)=\bigcup_{m\in N} F_m$. Indeed, let $z\in \overline{co}(A)$, there exists $C_z\in \R$ such that $p_n(z)\leq C_z \lambda_n$ for all $n\in \N$. If $C_z\leq 0$, then $z\in F_0$. If $C_z> 0$, it suffices to take $m_1:= [C_z]+1$ where $[C_z]$ denotes the floor of $C_z$ to have that $z\in F_{m_1}$. We deduce then that $F_m-a$ are closed and that $\overline{co}(A)-a=\bigcup_{m\in N} \left(F_m-a\right)$. Using the Baire Theorem on the complete metric space $(\overline{co}(A)-a,\|.\|_{F})$, we get an $m_0\in \N$ such that $F_{m_0}-a$ has no empty interior in $(\overline{co}(A)-a,\|.\|_{F})$.
Since by hypothesis $\overline{co}(A)-a$ has no empty interior in $F$, using Lemma \ref{Connue1} to obtain that $F_{m_0}-a$ has no empty interior in $F$. So there exists $z_0\in F_{m_0}-a$ and some integer number $N\in \N^*$ such that $B_{F}(z_0,\frac{1}{N})\subset F_{m_0}-a$. In other words, for all $z\in B_F(b, \frac{1}{N})\subset F_{m_0}$ (with $b:=a+z_0\in F_{m_0}\subset F$) and all $n\in \N$, we have:
\begin{eqnarray} \label{Q3}
p_n(z)\leq m_0 \lambda_n.
\end{eqnarray}
Let now $B$ a bounded subset of $F$, there exists an integer number $K_B\in \N^*$ such that $B\subset B_F(0,K_B)$. On the other hand, for all $h\in B$, there exists $z_h\in B_F(b,\frac{1}{N})$ such that $h=K_B.N(z_h-b)$. So using (\ref{Q3}) and the subadditivity of $p_n$, we obtain that, for all $n\in \N$:
\begin{eqnarray} 
p_n(h)= p_n(K_BN(z_h-b))&\leq& K_BNp_n(z_h-b)\nonumber\\
                                    &\leq& K_BN(p_n(z_h)+p_n(-b))\nonumber\\
                                    &\leq& K_BN m_0\lambda_n+ K_BNp_n(-b).\nonumber
\end{eqnarray}
On passing to the supremum on $B$, we obtain for all $n\in \N$,
$$\sup_{h\in B} p_n(h)\leq K_BN m_0\lambda_n+ K_BN p_n(-b).$$
Since $\lambda_n\longrightarrow 0$, we have
\begin{eqnarray*} 
\limsup_n \left(\sup_{h\in B} p_n(h) \right) &\leq& K_BN\limsup_n p_n(-b)\leq0.
\end{eqnarray*} 
This concludes the proof.
\end{proof}

\vskip5mm
As an immediat consequence, we obtain the following corollary. 

\begin{Cor} \label{Baire} Let $Z$ be a Banach space. Let $(f_n)_n\subset Z^*$ be a sequence of linear and continuous functionnals on $Z$ and let $(\lambda_n)_n\subset \R^+$ such that $\lambda_n\rightarrow 0$. Let $A$ be a no empty subset of $Z$, $a\in \overline{co}(A)$ and $F:=\overline{span(A-a)}$ the closed vector space generated by $A$. Suppose that $\overline{co}(A-a)$ ($=\overline{co}(A)-a$) has no empty interior in $F$ and that 
\begin{itemize}   
\item[$(1)$] for all $z\in \overline{co}(A)$, there exists a real number $C_z$ such that, for all $n\in \N$, we have $$f_n(z)\leq C_z \lambda_n.$$
\item[$(2)$] $f_n\stackrel{w^*}{\rightarrow}0$. 
\end{itemize}

Then, $\|(f_n)_{|F}\|_{F^*}\rightarrow 0$.
\end{Cor}

\begin{proof} The proof follows Lemma \ref{Baire1} with the subadditive and continuous maps $f_n$ and the bounded set $B:=S_{F^*}$.
\end{proof}
\vskip5mm
In the following corollary, the inequality in $ii)$ depends on $z\in Z$ unlike in \cite{LY} where the inequality is uniformly independent on $z$. Note also that if $C_z$ does not depend on $z$, the condition $ii)$ is also true by replacing: for all $z\in \overline{co}(Q)$ by for all $z\in Q$.
 
\begin{Cor} Let $Q\subset Z$ be a subset of finite codimension in $Z$. Let $(f_k)_k\subset Z^*$ and $\epsilon_k\geq 0$ and $\epsilon_k\rightarrow 0$ such that 
\begin{itemize}
\item[$i)$] $\|f_k\|\geq\delta> 0$, for all $k\in \N$, and $f_k\stackrel{w^*}{\rightarrow} f$.
\item[$ii)$] for all $z\in \overline{co}(Q)$, there exists $C_z\in \R$ such that for all $k\in \N$, $f_k(z)\geq -\epsilon_k C_z$.
\end{itemize}
Then, $f\neq 0$.
\end{Cor}
\begin{proof} Suppose by contradiction that $f=0$. By applying Corollary \ref{Baire} to $Q$ and $-f_k$, we obtain $\|(f_k)_{|F}\|_{F^*}\rightarrow 0$ where $F:=\overline{span(Q-z_0)}$. Since $F$ is of finite codimension in $Z$, there exists a finite-dimensional subspace $E$ of $Z$, such that $Z= F\oplus E$. Thus, there exists $L> 0$ such that $$\|f_k\|_{Z}\leq L\left(\|(f_k)_{|E}\|_{E^*}+\|(f_k)_{|F}\|_{F^*}\right).$$
Then, using $i)$ we obtain $\lim_k\|(f_k)_{|E}\|_{E^*}\geq \frac{\delta}{L}$. Since the weak-star topology and the norm topology coincids on $E$ because of finite dimension, we have that $0=\|(f)_{|E}\|_{E^*}=\lim_k \|(f_k)_{|E}\|_{E^*}\geq \frac{\delta}{L}>0$, which is a contradiction. Hence $f\neq 0$.
\end{proof}

\bibliographystyle{amsplain}

\end{document}